\documentclass[12pt]{article}

\usepackage{amsmath,amsthm,amssymb}
\usepackage[top=30truemm,bottom=30truemm,left=25truemm,right=25truemm]{geometry}
\usepackage[dvips]{graphicx,color,psfrag}
\usepackage{bm}

\theoremstyle{plain}
\newtheorem{definition}{Definition}[section]
\newtheorem{thm}[definition]{Theorem}

\newtheorem{lem}[definition]{Lemma}
\newtheorem{cor}[definition]{Corollary}
\newtheorem{rei}[definition]{Example}
\newtheorem{rem}[definition]{Remark}

\DeclareMathOperator{\sgn}{sgn}
\let\Re\relax
\DeclareMathOperator{\Re}{Re}

\title{Generalized Dedekind's theorem and its application to integer group determinants}
\author{Naoya Yamaguchi and Yuka Yamaguchi}
\date{\today}
\begin{document}

\maketitle

\begin{abstract}
In this paper, we give a refinement of a generalized Dedekind's theorem. 
In addition, we show that all possible values of integer group determinants of any group are also possible values of integer group determinants of its any abelian subgroup. 
By applying the refinement of a generalized Dedekind's theorem, 
we determine all possible values of integer group determinants of the direct product group of the cyclic group of order $8$ and the cyclic group of order $2$. 
\end{abstract}

\section{Introduction}
For a finite group $G$, 
let $x_{g}$ be an indeterminate for each $g \in G$ and let $\mathbb{Z}[x_{g}]$ be the multivariate polynomial ring in $x_{g}$ over $\mathbb{Z}$. 
The group matrix $M_{G}(x_{g})$ and the group determinant $\Theta_{G}(x_{g})$ of $G$ were defined by Dedekind as follows: 
$$
M_{G}(x_{g}) := \left( x_{g h^{- 1}} \right)_{g, h \in G}, \quad \Theta_{G}(x_{g}) := \det{M_{G}(x_{g})} \in \mathbb{Z}[x_{g}]. 
$$
When the elements of $G$ are reordered arbitrarily, 
the group matrix $M$ formed according to this reordering is of the form $M = P^{- 1} M_{G}(x_{g}) P$, 
where $P$ is an appropriate permutation matrix. 
Thus, $\Theta_{G}(x_{g})$ is invariant under any reordering of the elements of $G$. 
For a finite group $G$, 
let $\widehat{G}$ be a complete set of representatives of the equivalence classes of irreducible representations of $G$ over $\mathbb{C}$. 
Around 1880, 
for the case that $G$ is abelian, 
Dedekind gave the irreducible factorization of $\Theta_{G}(x_{g})$ over $\mathbb{C}$: 
{\it Let $G$ be a finite abelian group. 
Then}
$$
\Theta_{G}(x_{g}) = \prod_{\chi \in \widehat{G}} \sum_{g \in G} \chi(g) x_{g}. 
$$
This is called Dedekind's theorem. 
In 1896, 
Frobenius \cite{Frobenius1968gruppen} gave the irreducible factorization of $\Theta_{G}(x_{g})$ over $\mathbb{C}$ for any finite group: 
{\it Let $G$ be a finite group. Then }
$$
\Theta_{G}(x_{g}) = \prod_{\varphi \in \widehat{G}} \det{\left( \sum_{g \in G} \varphi(g) x_{g} \right)^{\deg{\varphi}}}. 
$$
This is the most well known generalization of Dedekind's theorem. 
This generalization is obtained from the decomposition of the regular representation $L$ of $G$ as a direct sum of irreducible representations and the expression $M_{G}(x_{g}) = \sum_{g \in G} x_{g} L(g)$. 
Frobenius created the character theory of finite groups in the process of obtaining the irreducible factorization. 
For the history on the theory, see, e.g., \cite{MR1659232, MR1554141, MR1554156, MR0446837, MR803326}. 
On the other hand, another generalization of Dedekind's theorem was given in \cite{MR3622295}: 
{\it Let $G$ be a finite abelian group and let $H$ be a subgroup of $G$. 
For every $h \in H$, there exists a homogeneous polynomial $A_{h} \in \mathbb{C}[x_{g}]$ satisfying $\deg{A_{h}} = \left| G / H \right|$ and 
\begin{align}
\Theta_{G}(x_{g}) = \prod_{\chi \in \widehat{H}} \sum_{h \in H} \chi(h) A_{h} = \Theta_{H}(A_{h}). 
\end{align}
If $H = G$, then we can take $A_{h} = x_{h}$ for each $h \in H$. }
This generalization shows that the group determinant of an abelian group can be written by the group determinant of any subgroup. 
Let ${\rm C}_{n} = \left\{ \overline{0}, \overline{1}, \ldots, \overline{n - 1} \right\}$ be the cyclic group of order $n$. 
The matrix $M_{{\rm C}_{n}}(x_{g})$ is similar to the circulant matrix of order $n$. 
That is, the circulant determinant is a special case of the group determinant. 
For the circulant determinant, 
Laquer \cite[Theorem~2]{MR624127} gave the following factorization in 1980: 
{\it 
Let $n = r s$, where $r$ and $s$ are relatively prime, and let $x_{j} := x_{\overline{j - 1}}$ for any $1 \leq j \leq n$. 
Then 
\begin{align*}
\Theta_{{\rm C}_{n}}{\left( x_{j} \right)} = \prod_{l = 0}^{s - 1} \Theta_{{\rm C}_{r}}{\left( y_{j}^{l} \right)}, \quad y_{j}^{l} := \sum_{k = 0}^{s - 1} \zeta_{s}^{l (k r + j - 1)} x_{k r + j}, 
\end{align*}
where $\zeta_{s}$ is a primitive $s$-th root of unity. 
}
We call this theorem Laquer's theorem. 
In recently, Laquer's theorem was generalized as follows \cite[Theorem~1.1]{Remark}: 
{\it
Let $G = H \times K$ be a direct product of finite abelian groups. 
Then we have }
\begin{align}
\Theta_{G}(x_{g}) = \prod_{\chi \in \widehat{K}} \Theta_{H}(y_{h}^{\chi}), \quad  y_{h}^{\chi} = \sum_{k \in K} \chi(k) x_{h k}. 
\end{align}

In this paper, we give a refinement of (1), 
which is a generalization of (2).

\begin{thm}\label{thm:1.1}
Let $G$ be a finite abelian group, 
let $H$ be a subgroup of $G$ and let 
\begin{align*}
\widehat{G}_{H} := \left\{ \chi \in \widehat{G} \mid \chi(h) = 1, \: h \in H \right\}, \quad G = \displaystyle\bigsqcup_{t \in T} t H, \quad \widehat{G} = \displaystyle\bigsqcup_{\chi \in X} \chi \widehat{G}_{H}. 
\end{align*}
Then we have 
\begin{align*}
\Theta_{G}(x_{g}) = \prod_{\chi \in X} \Theta_{G / H}{\left( y_{t H}^{\chi} \right)} = \Theta_{H}(z_{h}), 
\end{align*}
where 
$$
y_{t H}^{\chi} := \sum_{h \in H} \chi(t h) x_{t h}, \quad z_{h} := \frac{1}{| H |}\sum_{\chi \in X} \chi(h^{-1}) \Theta_{G/H}{\left( y_{t H}^{\chi} \right)}. 
$$
\end{thm}

For any $f(x_{g}) \in \mathbb{Z}[x_{g}]$, 
we denote by $f(x_{g})_{h}$ the sum of all monomials $c x_{g_{1}} x_{g_{2}} \cdots x_{g_{k}}$ in $f(x_{g})$ satisfying $g_{1} g_{2} \cdots g_{k} = h$. 
The following theorem gives another expression for $z_{h}$ in Theorem~$\ref{thm:1.1}$. 
When calculating $z_{h}$, 
the expression for $z_{h}$ in the following theorem might be more useful than one in Theorem~$\ref{thm:1.1}$ (see Example~$\ref{rei:2.3}$).

\begin{thm}\label{thm:1.2}
Let $y_{t H} := \sum_{h \in H} x_{t h}$. 
Then we have 
$$
z_{h} = \Theta_{G / H}{\left( y_{t H} \right)_{h}} \in \mathbb{Z}[x_{g}]. 
$$
\end{thm}

Theorem~$\ref{thm:1.1}$ is a refinement of (1) since $\left\{ \chi|_{H} \mid \chi \in X \right\} = \widehat{H}$ holds and we can take $A_{h} = z_{h}$ in (1). 
Note that for a finite abelian group $G$ and any subgroup $K$, 
there exists a subgroup $H$ of $G$ satisfying $K \cong G / H$. 
That is, Theorem~$\ref{thm:1.1}$ implies that $\Theta_{G}(x_{g})$ can also be expressed as a product of the group determinants of any subgroup. 
Thus, Theorem~$\ref{thm:1.1}$ derives (2). 
We apply Theorems~$\ref{thm:1.1}$ and $\ref{thm:1.2}$ to the study of the integer group determinant. 

A group determinant called an integer group determinant when its variables are integers. 
For a finite group $G$, let 
\begin{align*}
S(G) := \left\{ \Theta_{G}(x_{g}) \: | \: x_{g} \in \mathbb{Z} \right\}. 
\end{align*}
It immediately follows from $M_{G}(x_{g}) = \sum_{g \in G} x_{g} L(g)$ that $S(G)$ is a monoid. 
Determining $S(G)$ is an open problem. 
For the $G = {\rm C}_{n}$ cases, determining $S(G)$ is called Olga Taussky-Todd's circulant problem since Olga Taussky-Todd suggested it at the meeting of the American Mathematical Society in Hayward, California \cite{MR550657}. 
Even Olga Taussky-Todd's circulant problem remains as an open problem. 

For $S({\rm C}_{n})$, 
the following relation is known \cite[Lemma~3.6]{MR2914452}: 
{\it 
Let $n, q \geq 1$. 
If $q \mid n$, then }
\begin{align}
S({\rm C}_{n}) \subset S({\rm C}_{q}).
\end{align}
From Theorems~$\ref{thm:1.1}$ and $\ref{thm:1.2}$, 
we obtain a generalization of (3). 

\begin{cor}\label{cor:1.3}
Let $G$ be a finite abelian group and let $H$ be a subgroup of $G$. 
Then
$$
S(G) \subset S(H). 
$$
\end{cor}

Corollary~$\ref{cor:1.3}$ is generalized as follows. 

\begin{thm}\label{thm:1.4}
Let $G$ be a finite group and let $H$ be an abelian subgroup of $G$. 
Then
$$
\left\{ \alpha^{[G : H]} \mid \alpha \in S(H) \right\} \subset S(G) \subset S(H), 
$$
where $[G : H]$ is the index of $H$ in $G$. 
\end{thm}

For some types of groups, 
the problem was solved in \cite{MR3879399, MR2914452, MR624127, MR601702, MR550657, MR4363104, MR3998922, MR4056860, Integer16, https://doi.org/10.48550/arxiv.2209.12446}.  
As a result, 
for every group $G$ of order at most $15$, 
$S(G)$ is determined (see \cite{MR4363104, MR4056860}). 
For the groups of order $16$, 
the complete descriptions of $S(G)$ were obtained for ${\rm D}_{16}$ \cite[Theorem~5.3]{MR3879399}, 
${\rm C}_{16}$ \cite{Integer16} and ${\rm C}_{2}^{4}$ \cite{https://doi.org/10.48550/arxiv.2209.12446}, 
where ${\rm D}_{n}$ denotes the dihedral group of order $n$. 

Laquer \cite{MR624127} determined $S({\rm C}_{2 p})$, where $p$ is an odd prime, 
by using Laquer's theorem which provides an expression for the integer circulant determinant of ${\rm C}_{2 p}$ as a product of two integer circulant determinants of ${\rm C}_{p}$. 
%$$
%S(C_{2 p}) = \left\{ 2m + 1, \: 4 m \mid m \in \mathbb{Z} \right\} \cap \left\{ m', \: p^{2} m \mid \gcd(m', p) = 1, \: m, m' \in \mathbb{Z} \right\}, 
%$$
In \cite{https://doi.org/10.48550/arxiv.2209.12446}, 
$S({\rm C}_{2}^{4})$ is determined by using (2) which provides an expression for the integer group determinant of ${\rm C}_{2}^{n}$ as a product of two integer group determinants of ${\rm C}_{2}^{n - 1}$. 
We can generalize these approaches by using Theorem~$\ref{thm:1.1}$ to determine $S(G)$ for any abelian groups. 
There are fourteen groups of order $16$ up to isomorphism \cite{MR1510814, MR1505615}, 
and five of them are abelian. 
The unsolved abelian groups of order $16$ are ${\rm C}_{8} \times {\rm C}_{2}$, ${\rm C}_{4}^{2}$ and ${\rm C}_{4} \times {\rm C}_{2}^{2}$. 
By applying Theorem~$\ref{thm:1.1}$, 
we determine $S({\rm C}_{8} \times {\rm C}_{2})$. 

\begin{thm}\label{thm:1.5}
Let $A := \left\{ (8 k - 3) (8 l - 3) \mid k, l \in \mathbb{Z}, \: k \equiv l \pmod{2} \right\} \subsetneq \left\{ 16 m - 7 \mid m \in \mathbb{Z} \right\}$. 
Then we have 
\begin{align*}
S \left( {\rm C}_{8} \times {\rm C}_{2} \right) &= \left\{ 16 m + 1, \: m', \: 2^{10}(2 m + 1), \: 2^{12} m \mid m \in \mathbb{Z}, \: m' \in A \right\} \\ 
&\quad \cup \left\{ 2^{11} p (2 m + 1) \mid p = a^{2} + b^{2} \equiv 1, \: a + b \equiv \pm 3 \: \: ({\rm mod} \: {8}), \: m \in \mathbb{Z} \right\} \\ 
&\quad \cup \left\{ 2^{11} p (2 m + 1) \mid p \equiv - 3 \: \: ({\rm mod} \: {8}), \: m \in \mathbb{Z} \right\} \\ 
&\quad \cup \left\{ 2^{11} p^{2} (2 m + 1) \mid p \equiv 3 \: \: ({\rm mod} \: {8}), \: m \in \mathbb{Z} \right\}, 
\end{align*}
where $p$ denotes a prime. 
\end{thm}

The remaining two abelian groups could also be solved by using Theorem~$\ref{thm:1.1}$. (While this paper under review, it have been solved in \cite{https://doi.org/10.48550/arxiv.2211.14761, arXiv:2211.01597v2}. 
Also, as for non-abelian groups of order 16, ${\rm D}_{8} \times {\rm C}_{2}$, ${\rm Q}_{8} \times {\rm C}_{2}$ \cite[Theorems~3.1 and 4.1]{https://doi.org/10.48550/arxiv.2211.09930}, ${\rm Q}_{16}$ \cite{https://doi.org/10.48550/arxiv.2302.11688} and ${\rm C}_{2}^{2} \rtimes {\rm C}_{4}$ \cite{arXiv:2303.08489v2} have been solved, where ${\rm Q}_{n}$ denotes the generalized quaternion group of order $n$.)

Pinner and Smyth \cite[p.427]{MR4056860} noted the following inclusion relations for all groups of order~$8$: 
$$
S({\rm C}_{2}^{3}) \subsetneq S({\rm C}_{4} \times {\rm C}_{2}) \subsetneq S({\rm Q}_{8}) \subsetneq S({\rm D}_{8}) \subsetneq S({\rm C}_{8}). 
$$
 From preceding results and Theorem~$\ref{thm:1.5}$, 
 we have 
 $$
 S({\rm C}_{2}^{4}) \subsetneq S({\rm C}_{8} \times {\rm C}_{2}) \subsetneq S({\rm D}_{16}) \subsetneq S({\rm C}_{16}). 
 $$

Determining the integer group determinants aims to investigate the structure of a group by means of the group determinant. 
It is expected that individual new results will help us understand more about groups.

This paper is organized as follows. 
In Section~$\ref{sec:2}$, we prove Theorems~$\ref{thm:1.1}$ and $\ref{thm:1.2}$. 
In Sections~$\ref{sec:3}$ and $\ref{sec:4}$, we prove Theorems~$\ref{thm:1.4}$ and $\ref{thm:1.5}$, respectively.

\section{Proofs of Theorems~$\ref{thm:1.1}$ and $\ref{thm:1.2}$}\label{sec:2}
For a finite group $G$, 
let $x_{g}$ be an indeterminate for each $g \in G$, 
let $\mathbb{C}[x_{g}]$ be the multivariate polynomial ring in $x_{g}$ over $\mathbb{C}$, 
let $\mathbb{C} G$ the group algebra of $G$ over $\mathbb{C}$, 
and let $\mathbb{C}[x_{g}] G := \mathbb{C}[x_{g}] \otimes \mathbb{C} G = \left\{ \sum_{g \in G} A_{g} g \mid A_{g} \in \mathbb{C}[x_{g}] \right\}$ be the group algebra of $G$ over $\mathbb{C}[x_{g}]$. 
Also, 
for a finite abelian group $G$ and a subgroup $H$ of $G$, let 
$$
\widehat{G}_{H} := \left\{ \chi \in \widehat{G} \: \vert \: \chi(h) = 1, \: h \in H \right\}. 
$$
It is easily verified that 
$\widehat{G}_{H} = \left\{ \varphi \circ \pi \: \vert \: \varphi \in \widehat{G/H} \right\}$, 
where $\pi : G \rightarrow G/H$ is the canonical homomorphism. 
To prove Theorem~$\ref{thm:1.1}$, 
we use the following lemma. 

\begin{lem}[{\cite[Lemma~3.6]{MR3622295}}]\label{lem:2.1}
Let $G$ be a finite abelian group and $H$ be a subgroup of $G$. 
For every $h \in H$, there exists a homogeneous polynomial $A_{h} \in \mathbb{C}[x_{g}]$ satisfying $\deg{A_{h}} = \left| G / H \right|$ and 
$$
\prod_{\chi \in \widehat{G}_{H}} \sum_{g \in G} \chi(g) x_{g} g = \sum_{h \in H} A_{h} h \in \mathbb{C}[x_{g}]H. 
$$
If $H = G$, then we can take $A_{h} = x_{h}$ for each $h \in H$. 
\end{lem} 

\begin{proof}[Proof of Theorem~$\ref{thm:1.1}$]
From Dedekind's theorem, we have 
\begin{align*}
\Theta_{G}(x_{g}) 
&= \prod_{\chi \in \widehat{G}} \sum_{g \in G} \chi(g) x_{g} \\ 
&= \prod_{\chi \in X} \prod_{\chi' \in \widehat{G}_{H}} \sum_{t \in T} \sum_{h \in H} \left( \chi \chi' \right)(t h) x_{t h} \\ 
&= \prod_{\chi \in X} \prod_{\chi' \in \widehat{G}_{H}} \sum_{t \in T} \chi'(t) \sum_{h \in H} \chi(t h) x_{t h} \\ 
&= \prod_{\chi \in X} \prod_{\chi' \in \widehat{G/H}} \sum_{t H \in G/H} \chi'(t H) y_{t H}^{\chi} \\ 
&= \prod_{\chi \in X} \Theta_{G/H}{\left( y_{t H}^{\chi} \right)}. 
\end{align*}
Next, we show that for any $\chi \in X$, 
there exists $A_{h} \in \mathbb{C}[x_{g}]$ satisfying 
$$
\Theta_{G/H}{\left( y_{t H}^{\chi} \right)} = \sum_{h \in H} \chi(h) A_{h}. 
$$
For any $\chi \in X$, 
let $F_{\chi} \colon \mathbb{C}[x_{g}]G \to \mathbb{C}[x_{g}]G$ be the $\mathbb{C}[x_{g}]$-algebra homomorphism defined by $F_{\chi}(g) = \chi(g) g$. 
Then, from Lemma~$\ref{lem:2.1}$, 
there exists $A_{h} \in \mathbb{C}[x_{g}]$ satisfying 
\begin{align*}
\sum_{h \in H} \chi(h) A_{h} h 
&= F_{\chi} \left( \sum_{h \in H} A_{h} h \right) \\ 
&= F_{\chi} \left( \prod_{\chi' \in \widehat{G}_{H}} \sum_{g \in G} \chi'(g) x_{g} g \right) \\ 
%&= F_{\chi} \left( \prod_{\chi' \in \widehat{G}_{H}} \sum_{t \in T} \sum_{h \in H} \chi'(t h) x_{t h} t h \right) \\ 
&= F_{\chi} \left( \prod_{\chi' \in \widehat{G}_{H}} \sum_{t \in T} \sum_{h \in H} \chi'(t) x_{t h} t h \right) \\ 
&= \prod_{\chi' \in \widehat{G}_{H}} \sum_{t \in T} \sum_{h \in H}\chi'(t) \chi(t h)  x_{t h} t h. 
\end{align*}
Let $F \colon \mathbb{C}[x_{g}]G \to \mathbb{C}[x_{g}]$ be the $\mathbb{C}[x_{g}]$-algebra homomorphism defined by $F(g) = 1$. 
Applying $F$ to the both sides of the above, 
we have 
$$
\sum_{h \in H} \chi(h) A_{h} = \prod_{\chi' \in \widehat{G}_{H}} \sum_{t \in T} \sum_{h \in H} \chi'(t) \chi(t h) x_{t h} = \prod_{\chi' \in \widehat{G/H}} \sum_{t H \in G/H} \chi'(t H) y_{t H}^{\chi} = \Theta_{G/H}{\left( y_{t H}^{\chi} \right)}. 
$$
From the above, it follows that there exists $A_{h} \in \mathbb{C}[x_{g}]$ satisfying 
\begin{align}
\Theta_{G}(x_{g}) = \prod_{\chi \in X} \Theta_{G/H}{\left( y_{t H}^{\chi} \right)} = \prod_{\chi \in X} \sum_{h \in H} \chi(h) A_{h} = \prod_{\chi \in \widehat{H}} \sum_{h \in H} \chi(h) A_{h} = \Theta_{H}(A_{h}). 
\end{align}
Finally, we show that $A_{h}$ in (4) is expressed as 
$$
A_{h} = \frac{1}{| H |}\sum_{\chi \in X} \chi(h^{-1}) \Theta_{G/H}{\left( y_{t H}^{\chi} \right)} 
$$
for any $h \in H$. 
From orthogonality relations for characters, 
for any $h \in H$, we have 
\begin{align*}
\sum_{\chi \in X} \chi(h^{-1}) \Theta_{G/H}{\left( y_{t H}^{\chi} \right)} 
&= \sum_{\chi \in X} \chi(h^{-1}) \sum_{h' \in H} \chi(h') A_{h'} \\ 
&= \sum_{\chi \in X} \sum_{h' \in H} \chi(h^{-1} h') A_{h'} \\ 
&= \sum_{h' \in H} \sum_{\chi \in \widehat{H}} \chi(h^{-1} h') A_{h'} \\ 
&= | H | A_{h}. 
\end{align*}
\end{proof}

\begin{rem}\label{rem:2.2}
From the proof of Theorem~$\ref{thm:1.1}$, 
$A_{h}$ in Lemma~$\ref{lem:2.1}$ equals to $z_{h}$ in Theorem~$\ref{thm:1.1}$. 
\end{rem}

%To prove Theorem~$\ref{thm:1.2}$, 
%it should be mentioned that $A_{h}$ satisfying Lemma~$\ref{lem:2.1}$ also satisfy~(1). 
%For details, see the proof of \cite[Theorem~3.8]{MR3622295}. 

\begin{proof}[Proof of Theorem~$\ref{thm:1.2}$]
From Lemma~$\ref{lem:2.1}$ and Remark~$\ref{rem:2.2}$, 
we have 
\begin{align*}
\sum_{h \in H} z_{h} h = \prod_{\chi \in \widehat{G}_{H}} \sum_{g \in G} \chi(g) x_{g} g = \prod_{\chi \in \widehat{G / H}} \sum_{g \in G} \chi(g H) x_{g} g = \prod_{\chi \in \widehat{G / H}} \sum_{t H \in G / H} \sum_{h' \in H} \chi(t H) x_{t h'} t h'. 
\end{align*}
Therefore, we have 
\begin{align*}
z_{h} = \left( \prod_{\chi \in \widehat{G / H}} \sum_{t H \in G / H} \chi(t H) \sum_{h' \in H} x_{t h'} \right)_{h} = \Theta_{G / H}{\left( y_{t H} \right)}_{h} \in \mathbb{Z}[x_{g}]. 
\end{align*}
\end{proof}

From Theorems~$\ref{thm:1.1}$ and $\ref{thm:1.2}$, 
we can take $z_{h} \in \mathbb{Z}[x_{g}]$ satisfying $\Theta_{G}(x_{g}) = \Theta_{H}(z_{h})$. 
Thus, Corollary~$\ref{cor:1.3}$ is obtained.

\begin{rei}\label{rei:2.3}
Using Theorems~$\ref{thm:1.1}$ and $\ref{thm:1.2}$, 
we calculate $\Theta_{{\rm C}_{4}}(x_{g})$. 
Let $G = {\rm C}_{4}$ and $H = \{ \overline{0}, \overline{2} \}$. 
Then, $G / H = \left\{ \overline{0} H, \overline{1} H \right\}$. 
We write $x_{\overline{i}}$ as $x_{i}$ for any $0 \leq i \leq 3$. 
From Theorem~$\ref{thm:1.2}$, 
we have 
\begin{align*}
z_{\overline{0}} = \Theta_{G/H}(y_{t H})_{\overline{0}} = x_{0}^{2} + x_{2}^{2} - 2 x_{1} x_{3}, \quad 
z_{\overline{2}} = \Theta_{G/H}(y_{t H})_{\overline{2}} = 2 x_{0} x_{2} - x_{1}^{2} - x_{3}^{2} 
\end{align*}
since $y_{\overline{0} H} = x_{0} + x_{2}$, $y_{\overline{1} H} = x_{1} + x_{3}$ and 
\begin{align*}
\Theta_{G/H}(y_{t H}) = y_{\overline{0} H}^{2} - y_{\overline{1} H}^{2} = \left( x_{0}^{2} + 2 x_{0} x_{2} + x_{2}^{2} \right) - \left( x_{1}^{2} + 2 x_{1} x_{3} + x_{3}^{2} \right). 
\end{align*}
Therefore, from Theorem~$\ref{thm:1.1}$, 
we have 
\begin{align*}
\Theta_{G}(x_{g}) = \Theta_{H}(z_{h}) = z_{\overline{0}}^{2} - z_{\overline{2}}^{2} = \left( x_{0}^{2} + x_{2}^{2} - 2 x_{1} x_{3} \right)^{2} - \left( 2 x_{0} x_{2} - x_{1}^{2} - x_{3}^{2} \right)^{2}. 
\end{align*}
\end{rei}

\section{Proof of Theorem~$\ref{thm:1.4}$}\label{sec:3}

The lower bound in Theorem~$\ref{thm:1.4}$ is derived from \cite[Lemma~3.2]{MR4227663}. 
Also, the upper bound immediately follows from the following lemma essentially provided in \cite[Theorem~1.4]{MR4031482}. 

\begin{lem}\label{lem:3.1}
Let $G$ be a finite group and let $H$ be an abelian subgroup of $G$. 
Then, there exists a homogeneous polynomial $A_{h} \in \mathbb{Z}[x_{g}]$ satisfying $\deg{A_{h}} = \left[ G : H \right]$ and 
$$
\Theta_{G}(x_{g}) = \Theta_{H}(A_{h}), 
$$
where $[G : H]$ is the index of $H$ in $G$. 
\end{lem}
The proof in \cite{MR4031482} is not concise. 
We give a brief proof of Lemma~$\ref{lem:3.1}$.
For the purpose, we use the following identity \cite[p.~82, Theorem~2.6]{johnson2019group}; see also \cite{MR1563590, MR1732497}: 
{\it Given the block matrix $M$ of the form 
$$
\begin{pmatrix}
M_{1 1} & M_{1 2} & \cdots & M_{1 n} \\ 
M_{2 1} & M_{2 2} & \cdots & M_{2 n} \\ 
\vdots & \vdots & \ddots & \vdots \\ 
M_{n 1} & M_{n 2} & \cdots & M_{n n}
\end{pmatrix}, 
$$
where the matrices $M_{i j}$ are pairwise commuting of size $m \times m$ then 
$$
\det{M} = \det{\left( \sum_{\sigma \in S_{n}} \sgn(\sigma) M_{1 \sigma(1)} M_{2 \sigma(2)} \cdots M_{n \sigma(n)} \right)}. 
$$
}

\begin{proof}[Proof of Lemma~$\ref{lem:3.1}$]
Let $G = \{ g_{1}, g_{2}, \ldots, g_{m n} \}$, let $H = \{ h_{1}, h_{2}, \ldots, h_{m} \}$ and let 
$G = t_{1} H \sqcup t_{2} H \sqcup \cdots \sqcup t_{n} H$, 
where $g_{i} = t_{k} h_{l} \in G$ with $i = (k - 1) m + l$ for $1 \leq k \leq n$ and $1 \leq l \leq m$. 
Then, 
the group matrix $\left( x_{g_{i} g_{j}^{- 1}} \right)_{1 \leq i, j \leq m n}$ of $G$ can be expressed as the block matrix: 
$$
\left( x_{g_{i} g_{j}^{- 1}} \right)_{1 \leq i, j \leq m n} = \left( M_{k l} \right)_{1 \leq k, l \leq n}, 
$$
where $M_{k l}$ is the matrix obtained by replacing each $x_{h_{i} h_{j}^{-1}}$ in the group matrix $\left( x_{h_{i} h_{j}^{- 1}} \right)_{1 \leq i, j \leq m}$ of $H$ to $x_{(t_{k} h_{i}) (t_{l} h_{j})^{-1}}$. 
That is, $M_{k l} = \left( x_{(t_{k} h_{i}) (t_{l} h_{j})^{-1}} \right)_{1 \leq i, j \leq m}$. 
Since $H$ is abelian, $M_{k l}$ are pairwise commuting. 
Therefore, there exists $A_{h} \in \mathbb{Z}[x_{g}]$ satisfying 
$$
\Theta_{G}(x_{g}) = \det{\left( x_{g_{i} g_{j}^{-1}} \right)_{1 \leq i, j \leq m n}} = \det{\left( \sum_{\sigma \in S_{n}} \sgn(\sigma) M_{1 \sigma(1)} M_{2 \sigma(2)} \cdots M_{n \sigma(n)} \right)} = \Theta_{H}(A_{h}) 
$$
since $\sum_{\sigma \in S_{n}} \sgn(\sigma) M_{1 \sigma(1)} M_{2 \sigma(2)} \cdots M_{n \sigma(n)}$ is also of the form of a group matrix of $H$. 
\end{proof}

\section{Proof of Theorem~$\ref{thm:1.5}$}\label{sec:4}
In this section, 
by applying Theorem~$\ref{thm:1.1}$, 
we determine $S({\rm C}_{8} \times {\rm C}_{2})$. 

\subsection{Relations with group determinants of subgroups}

We denote the variable $x_{\overline{i}}$ by $x_{i}$ for any $\overline{i} \in {\rm C}_{n}$ and let $D_{n}(x_{0}, x_{1}, \ldots, x_{n - 1}) := \Theta_{{\rm C}_{n}}(x_{g})$. 
Also, for any $g = ( \overline{r}, \overline{s} ) \in {\rm C}_{8} \times {\rm C}_{2}$ with $r \in \{ 0, 1, \ldots, 7 \}$ and $s \in \{ 0, 1 \}$, 
we denote the variable $y_{g}$ by $y_{j}$, where $j := r + 8s$, 
and let $D_{8 \times 2}(y_{0}, y_{1}, \ldots, y_{15}) := \Theta_{{\rm C}_{8} \times {\rm C}_{2}}(y_{g})$. 
From the $G = {\rm C}_{8} \times {\rm C}_{2}$ and $H = \left\{ (\overline{0}, \overline{0}), (\overline{0}, \overline{1}) \right\}$ case of Theorem~$\ref{thm:1.1}$, we have 
\begin{align*}
D_{8 \times 2}(y_{0}, \ldots, y_{15}) = D_{8}(y_{0} + y_{8}, \ldots, y_{7} + y_{15}) D_{8}(y_{0} - y_{8}, \ldots, y_{7} - y_{15}). 
\end{align*}
Let $\zeta_{n}$ be a primitive $n$-th root of unity. 
From the $G = {\rm C}_{8}$ and $H = \left\{ \overline{0}, \overline{4} \right\}$ case of Theorem~$\ref{thm:1.1}$, we have 
\begin{align*}
D_{8}(x_{0}, x_{1}, \ldots, x_{7}) 
&= D_{4}(x_{0} + x_{4}, x_{1} + x_{5}, x_{2} + x_{6}, x_{3} + x_{7}) \\ 
&\quad \times D_{4}(x_{0} - x_{4}, \zeta_{8}(x_{1} - x_{5}), \zeta_{8}^{2}(x_{2} - x_{6}), \zeta_{8}^{3}(x_{3} - x_{7})). 
\end{align*}
From the $G = {\rm C}_{4}$ and $H = \left\{ \overline{0}, \overline{2} \right\}$ case of Theorem~$\ref{thm:1.1}$, we have 
\begin{align*}
D_{4}(x_{0}, x_{1}, x_{2}, x_{3}) 
&= D_{2}(x_{0} + x_{2}, x_{1} + x_{3}) D_{2}(x_{0} - x_{2}, \zeta_{4} (x_{1} - x_{3})) \\ 
&= D_{2}(x_{0}^{2} + x_{2}^{2} - 2 x_{1} x_{3}, - x_{1}^{2} - x_{3}^{2} + 2 x_{0} x_{2}). 
\end{align*}

Let $\widetilde{D}_{4}(x_{0}, x_{1}, x_{2}, x_{3}) := D_{4}(x_{0}, \zeta_{8} x_{1}, \zeta_{8}^{2} x_{2}, \zeta_{8}^{3} x_{3})$. 
Then we have the following lemma. 

\begin{lem}\label{lem:4.1}
The following hold: 
\begin{enumerate}
\item[$(1)$] $D_{4}(x_{0}, x_{1}, x_{2}, x_{3}) = \left\{ (x_{0} + x_{2})^{2} - (x_{1} + x_{3})^{2} \right\} \left\{ (x_{0} - x_{2})^{2} + (x_{1} - x_{3})^{2} \right\}$; 
\item[$(2)$] $\widetilde{D}_{4}(x_{0}, x_{1}, x_{2}, x_{3}) = (x_{0}^{2} - x_{2}^{2} + 2 x_{1} x_{3})^{2} + (x_{1}^{2} - x_{3}^{2} - 2 x_{0} x_{2})^{2}$. 
\end{enumerate}
\end{lem}

Lemma~$\ref{lem:4.1}$~(2) shows that $\widetilde{D}_{4}(x_{0}, x_{1}, x_{2}, x_{3}) \in \mathbb{Z}$ holds for any $x_{0}, x_{1}, x_{2}, x_{3} \in \mathbb{Z}$. 
Throughout this paper, 
we assume that $a_{0}, a_{1}, \ldots, a_{15} \in \mathbb{Z}$, and for any $0 \leq i \leq 3$, put 
\begin{align*}
b_{i} &:= (a_{i} + a_{i + 8}) + (a_{i + 4} + a_{i + 12}), & c_{i} &:= (a_{i} + a_{i + 8}) - (a_{i + 4} + a_{i + 12}), \\ 
d_{i} &:= (a_{i} - a_{i + 8}) + (a_{i + 4} - a_{i + 12}), & e_{i} &:= (a_{i} - a_{i + 8}) - (a_{i + 4} - a_{i + 12}). 
\end{align*}
Also, let $\bm{a} := (a_{0}, a_{1}, \ldots, a_{15})$ and let 
\begin{align*}
\bm{b} := (b_{0}, b_{1}, b_{2}, b_{3}), \quad 
\bm{c} := (c_{0}, c_{1}, c_{2}, c_{3}), \quad 
\bm{d} := (d_{0}, d_{1}, d_{2}, d_{3}), \quad 
\bm{e} := (e_{0}, e_{1}, e_{2}, e_{3}). 
\end{align*}
The following relations will be frequently used in this paper: 
\begin{align*}
D_{8 \times 2}(\bm{a}) 
&= D_{8}(a_{0} + a_{8}, a_{1} + a_{9}, \ldots, a_{7} + a_{15}) D_{8}(a_{0} - a_{8}, a_{1} - a_{9}, \ldots, a_{7} - a_{15}) \\ 
&= D_{4}(\bm{b}) \widetilde{D}_{4}(\bm{c}) D_{4}(\bm{d}) \widetilde{D}_{4}(\bm{e}). 
\end{align*}

\begin{rem}\label{rem:4.2}
For any $0 \leq i \leq 3$, the following hold: 
\begin{enumerate}
\item[$(1)$] $b_{i} \equiv c_{i} \equiv d_{i} \equiv e_{i} \pmod{2}$; 
\item[$(2)$] $b_{i} + c_{i} + d_{i} + e_{i} \equiv 0 \pmod{4}$. 
\end{enumerate}
\end{rem}

\begin{lem}\label{lem:4.3}
We have $D_{8 \times 2}(\bm{a}) \equiv D_{4}(\bm{b}) \equiv \widetilde{D}_{4}(\bm{c}) \equiv D_{4}(\bm{d}) \equiv \widetilde{D}_{4}(\bm{e}) \pmod{2}$. 
\end{lem}
\begin{proof}
From Lemma~$\ref{lem:4.1}$, 
we have $D_{4}(x_{0}, x_{1}, x_{2}, x_{3}) \equiv x_{0} + x_{1} + x_{2} + x_{3} \equiv \widetilde{D}_{4}(x_{0}, x_{1}, x_{2}, x_{3}) \pmod{2}$. 
Therefore, from Remark~$\ref{rem:4.2}$~(1), the lemma is proved. 
\end{proof}

\subsection{Impossible odd numbers}

Let $\mathbb{Z}_{\rm odd}$ be the set of all odd numbers and $A := \left\{ (8 k - 3) (8 l - 3) \mid k, l \in \mathbb{Z}, \: k \equiv l \pmod{2} \right\}$. 

\begin{lem}\label{lem:5.1}
We have $S({\rm C}_{8} \times {\rm C}_{2}) \cap \mathbb{Z}_{\rm odd} \subset \left\{ 16 m + 1 \mid m \in \mathbb{Z} \right\} \cup A$. 
\end{lem}

To prove Lemma~$\ref{lem:5.1}$, 
we use the following three lemmas.

\begin{lem}\label{lem:5.2}
We have 
\begin{align*}
&D_{4}(\bm{b}) \widetilde{D}_{4}(\bm{c}) D_{4}(\bm{d}) \widetilde{D}_{4}(\bm{e}) \\ 
&\quad= D_{4}(b_{1}, b_{2}, b_{3}, b_{0}) \widetilde{D}_{4}(c_{1}, c_{2}, c_{3}, - c_{0}) D_{4}(d_{1}, d_{2}, d_{3}, d_{0}) \widetilde{D}_{4}(e_{1}, e_{2}, e_{3}, - e_{0}) \\ 
&\quad= D_{4}(b_{2}, b_{3}, b_{0}, b_{1}) \widetilde{D}_{4}(c_{2}, c_{3}, - c_{0}, - c_{1}) D_{4}(d_{2}, d_{3}, d_{0}, d_{1}) \widetilde{D}_{4}(e_{2}, e_{3}, - e_{0}, - e_{1}) \\ 
&\quad= D_{4}(b_{3}, b_{0}, b_{1}, b_{2}) \widetilde{D}_{4}(c_{3}, - c_{0}, - c_{1}, - c_{2}) D_{4}(d_{3}, d_{0}, d_{1}, d_{2}) \widetilde{D}_{4}(e_{3}, - e_{0}, - e_{1}, - e_{2}). 
\end{align*}
\end{lem}
\begin{proof}
From Lemma~$\ref{lem:4.1}$, we have 
\begin{align*}
D_{4}(x_{0}, x_{1}, x_{2}, x_{3}) = - D_{4}(x_{1}, x_{2}, x_{3}, x_{0}), \quad \widetilde{D}_{4}(x_{0}, x_{1}, x_{2}, x_{3}) &= \widetilde{D}_{4}(x_{1}, x_{2}, x_{3}, - x_{0}). 
\end{align*}
Therefore, the lemma is proved. 
\end{proof}

\begin{lem}\label{lem:5.3}
For any $k, l, m, n \in \mathbb{Z}$, 
the following hold: 
\begin{enumerate}
\item[$(1)$] $D_{4}(2 k + 1, 2 l, 2 m, 2 n) \equiv 8 m + 1 \pmod{16}$; 
\item[$(2)$] $\widetilde{D}_{4}(2 k + 1, 2 l, 2 m, 2 n) \equiv 8 m + 1 \pmod{16}$. 
\end{enumerate}
\end{lem}
\begin{proof}
Let $D := D_{4}(2 k + 1, 2 l, 2 m, 2 n)$ and $\widetilde{D} := \widetilde{D}_{4}(2 k + 1, 2 l, 2 m, 2 n)$. 
Then we have 
\begin{align*}
D 
%&= \left\{ (2 k + 1 + 2m)^{2} - (2 l + 2 n)^{2} \right\} \left\{ (2 k + 1 - 2 m)^{2} + (2 l - 2 n)^{2} \right\} \\ 
&= \left\{ 4 (k + m)^{2} + 4 (k + m) + 1 - 4 (l + n)^{2} \right\} \left\{ 4 (k - m)^{2} + 4 (k - m) + 1 + 4 (l - n)^{2} \right\} \\ 
%&\equiv 4 (k + m)^{2} + 4 (k + m) - 4 (l + n)^{2} + 4 (k - m)^{2} + 4 (k - m) + 4 (l - n)^{2} + 1 \\ 
%&\equiv 8 k (k + 1) + 8 m^{2} + 1 \\ 
&\equiv 8 m + 1 \pmod{16}, \\ 
\widetilde{D} 
%&= \left\{ (2 k + 1)^{2} - (2 m)^{2} + 2 (2 l) (2 n) \right\}^{2} + \left\{ (2 l)^{2} - (2 n)^{2} - 2 (2 k + 1) (2 m) \right\}^{2} \\ 
&= \left\{ 4 k (k + 1) + 1 - 4 m^{2} + 8 l n \right\}^{2} + \left\{ 4 l^{2} - 4 n^{2} - 8 k m - 4 m \right\}^{2} \\ 
%&\equiv 1 - 8 m^{2} \\ 
&\equiv 8 m + 1 \pmod{16}. 
\end{align*}
\end{proof}

\begin{lem}\label{lem:5.4}
For any $k, l, m, n \in \mathbb{Z}$, 
the following hold: 
\begin{enumerate}
\item[$(1)$] $D_{4}(2 k, 2 l + 1, 2 m + 1, 2 n + 1) \equiv 8 (k + l + n) - 3 \pmod{16}$; 
\item[$(2)$] $\widetilde{D}_{4}(2 k, 2 l + 1, 2 m + 1, 2 n + 1) \equiv 8 (k + l + n) + 1 \pmod{16}$. 
\end{enumerate}
\end{lem}
\begin{proof}
Let $D := D_{4}(2 k, 2 l + 1, 2 m + 1, 2 n + 1)$ and $\widetilde{D} := \widetilde{D}_{4}(2 k, 2 l + 1, 2 m + 1, 2 n + 1)$. 
Then, 
\begin{align*}
D 
%&= \left\{ (2 k + 2 m + 1)^{2} - (2 l + 2 n + 2)^{2} \right\} \left\{ (2 k - 2 m - 1)^{2} + (2 l - 2 n)^{2} \right\} \\ 
&= \left\{ 4 (k + m)^{2} + 4 (k + m) + 1 - 4 (l + n + 1)^{2} \right\} \left\{ 4 (k - m)^{2} - 4 (k - m) + 1 + 4 (l - n)^{2} \right\} \\ 
%&\equiv 4 (k + m)^{2} + 4 (k + m) - 4 (l + n + 1)^{2} + 4 (k - m)^{2} - 4 (k - m) + 4 (l - n)^{2} + 1 \\ 
%&\equiv 8 k^{2} + 8 m (m + 1) - 8 l - 8 n - 3 \\ 
&\equiv 8 (k + l + n) - 3 \pmod{16}, \\ 
\widetilde{D} 
%&= \left\{ (2 k)^{2} - (2 m + 1)^{2} + 2 (2 l + 1) (2 n + 1) \right\}^{2} + \left\{ (2 l + 1)^{2} - (2 n + 1)^{2} - 2 (2 k) (2 m + 1) \right\}^{2} \\ 
&= \left\{ 4 k^{2} - 4 m (m + 1) + 8 l n + 4 l + 4 n + 1 \right\}^{2} + \left\{ 4 l (l + 1) - 4 n (n + 1) - 8 k m - 4 k \right\}^{2} \\ 
%&\equiv 8 k^{2} + 8 l + 8 n + 1 \\ 
&\equiv 8 (k + l + n) + 1 \pmod{16}. 
\end{align*}
\end{proof}

\begin{proof}[Proof of Lemma~$\ref{lem:5.1}$]
Let $D_{8 \times 2}(\bm{a}) = D_{4}(\bm{b}) \widetilde{D}_{4}(\bm{c}) D_{4}(\bm{d}) \widetilde{D}_{4}(\bm{e}) \in \mathbb{Z}_{\rm odd}$. 
Then, $b_{0} + b_{2} \not\equiv b_{1} + b_{3} \pmod{2}$ holds since $D_{4}(\bm{b})$ is odd. 
We prove the following: 
\begin{enumerate}
\item[(i)] If exactly three of $b_{0}, b_{1}, b_{2}, b_{3}$ are even, then $D_{8 \times 2}(\bm{a}) \in \left\{ 16 m + 1 \mid m \in \mathbb{Z} \right\}$; 
\item[(ii)] If exactly one of $b_{0}, b_{1}, b_{2}, b_{3}$ is even, then $D_{8 \times 2}(\bm{a}) \in A$. 
\end{enumerate}
First, we prove (i). 
%If exactly three of $b_{0}, b_{1}, b_{2}, b_{3}$ are even, then 
%$$
%\bm{b} \equiv (1, 0, 0, 0), \: (0, 0, 0, 1), \: (0, 0, 1, 0) \; \text{or} \; (0, 1, 0, 0) \pmod{2}. 
%$$
If  $\bm{b} \equiv (1, 0, 0, 0) \pmod{2}$, 
then there exist $m_{i} \in \mathbb{Z}$ satisfying $b_{2} = 2 m_{0}$, $c_{2} = 2 m_{1}$, $d_{2} = 2 m_{2}$, $e_{2} = 2 m_{3}$ and $\sum_{i = 0}^{3} m_{i} \equiv 0 \pmod{2}$ from Remark~$\ref{rem:4.2}$. 
Therefore, from Lemma~$\ref{lem:5.3}$, 
$D_{8 \times 2}(\bm{a}) \equiv \prod_{i = 0}^{3} (8 m_{i} + 1) \equiv 8 \sum_{i = 0}^{3} m_{i} + 1 \equiv 1 \pmod{16}$. 
%\begin{align*}
%D_{8 \times 2}(\bm{a}) 
%\equiv \prod_{i = 0}^{3} (8 m_{i} + 1) 
%\equiv 8 \sum_{i = 0}^{3} m_{i} + 1 
%\equiv 1 \pmod{16}. 
%&\equiv (8 m_{0} + 1) (8 m_{1} + 1) (8 m_{2} + 1) (8 m_{3} + 1)  \\ 
%&\equiv 8 (m_{0} + m_{1} + m_{2} + m_{3}) + 1 \\ 
%&\equiv 1 \pmod{16}. 
%\end{align*}
From this and Lemma~$\ref{lem:5.2}$, 
the remaining three cases are also proved. 
Next, we prove (ii). 
%If exactly one of $b_{0}, b_{1}, b_{2}, b_{3}$ is even, then 
%$$
%\bm{b} \equiv (0, 1, 1, 1), \: (1, 1, 1, 0), \: (1, 1, 0, 1) \; \text{or} \; (1, 0, 1, 1) \pmod{2}. 
%$$
If  $\bm{b} \equiv (0, 1, 1, 1) \pmod{2}$, 
then there exist $k_{i}, l_{i}, n_{i} \in \mathbb{Z}$ satisfying 
$(b_{0}, b_{1}, b_{3}) = (2 k_{0}, 2 l_{0} + 1, 2 n_{0} + 1)$, $(c_{0}, c_{1}, c_{3}) = (2 k_{1}, 2 l_{1} + 1, 2 n_{1} + 1)$, 
$(d_{0}, d_{1}, d_{3}) = (2 k_{2}, 2 l_{2} + 1, 2 n_{2} + 1)$, $(e_{0}, e_{1}, e_{3}) = (2 k_{3}, 2 l_{3} + 1, 2 n_{3} + 1)$ 
%\begin{align*}
%(b_{0}, b_{1}, b_{3}) &= (2 k_{0}, 2 l_{0} + 1, 2 n_{0} + 1), & (c_{0}, c_{1}, c_{3}) &= (2 k_{1}, 2 l_{1} + 1, 2 n_{1} + 1), \\ 
%(d_{0}, d_{1}, d_{3}) &= (2 k_{2}, 2 l_{2} + 1, 2 n_{2} + 1), & (e_{0}, e_{1}, e_{3}) &= (2 k_{3}, 2 l_{3} + 1, 2 n_{3} + 1) 
%\end{align*}
and $\sum_{i = 0}^{3} k_{i} \equiv \sum_{i = 0}^{3} l_{i} \equiv \sum_{i = 0}^{3} n_{i} \equiv 0 \pmod{2}$ from Remark~$\ref{rem:4.2}$. 
Therefore, from Lemma~$\ref{lem:5.4}$, we have 
$D_{4}(\bm{b}) \widetilde{D}_{4}(\bm{c}) \equiv (8 r_{0} - 3) (8 r_{1} + 1) \equiv 8 r_{0} + 8 r_{1} - 3 \pmod{16}$ and  
$D_{4}(\bm{d}) \widetilde{D}_{4}(\bm{e}) \equiv (8 r_{2} - 3) (8 r_{3} + 1) \equiv 8 r_{2} + 8 r_{3} - 3 \pmod{16}$, 
%\begin{align*}
%D_{4}(\bm{b}) \widetilde{D}_{4}(\bm{c}) 
%&\equiv (8 r_{0} - 3) (8 r_{1} + 1) \equiv 8 r_{0} + 8 r_{1} - 3 \pmod{16}, \\ 
%D_{4}(\bm{d}) \widetilde{D}_{4}(\bm{e}) 
%&\equiv (8 r_{2} - 3) (8 r_{3} + 1) \equiv 8 r_{2} + 8 r_{3} - 3 \pmod{16}, 
%\end{align*}
where $r_{i} := k_{i} + l_{i} + n_{i}$. 
Thus, there exist $s_{0}, s_{1} \in \mathbb{Z}$ satisfying 
$D_{4}(\bm{b}) \widetilde{D}_{4}(\bm{c}) = 16 s_{0} + 8 r_{0} + 8 r_{1} - 3$, 
$D_{4}(\bm{d}) \widetilde{D}_{4}(\bm{e}) = 16 s_{1} + 8 r_{2} + 8 r_{3} - 3$. 
Let $k := 2 s_{0} + r_{0} + r_{1}$ and $l := 2 s_{1} + r_{2} + r_{3}$. 
Then $D_{8 \times 2}(\bm{a}) = (8 k - 3) (8 l - 3) \in A$ since $k \equiv l \pmod{2}$ holds from $\sum_{i = 0}^{3} r_{i} \equiv 0 \pmod{2}$. 
From this and Lemma~$\ref{lem:5.2}$, 
the remaining three cases are also proved. 
\end{proof}

\subsection{Impossible even numbers}

We will use Kaiblinger's \cite[Theorem~1.1]{MR2914452} results 
$S({\rm C}_{4}) = \mathbb{Z}_{\rm odd} \cup 2^{4} \mathbb{Z}$ and $S({\rm C}_{8}) = \mathbb{Z}_{\rm odd} \cup 2^{5} \mathbb{Z}$. 
%Kaiblinger \cite[Theorem~1.1]{MR2914452} gave the complete descriptions of $S({\rm C}_{4})$ and $S({\rm C}_{8})$ as follows: 
%$$
%S({\rm C}_{4}) = \mathbb{Z}_{\rm odd} \cup 2^{4} \mathbb{Z}, \quad S({\rm C}_{8}) = \mathbb{Z}_{\rm odd} \cup 2^{5} \mathbb{Z}. 
%$$

\begin{lem}\label{lem:6.1}
We have $S \left( {\rm C}_{8} \times {\rm C}_{2} \right) \cap 2 \mathbb{Z} \subset 2^{10} \mathbb{Z}$. 
\end{lem}
\begin{proof}
Let $D_{8 \times 2}(\bm{a}) = D_{8}(a_{0} + a_{8}, \ldots, a_{7} + a_{15}) D_{8}(a_{0} - a_{8}, \ldots, a_{7} - a_{15}) \in 2 \mathbb{Z}$. 
Since $D_{8}(a_{0} + a_{8}, \ldots, a_{7} + a_{15}) \equiv D_{8}(a_{0} - a_{8}, \ldots, a_{7} - a_{15}) \pmod{2}$ 
holds from $a_{i} + a_{i + 8} \equiv a_{i} - a_{i + 8} \pmod{2}$, 
we have $D_{8}(a_{0} + a_{8}, \ldots, a_{7} + a_{15}), D_{8}(a_{0} - a_{8}, \ldots, a_{7} - a_{15}) \in S({\rm C}_{8}) \cap 2 \mathbb{Z} = 2^{5} \mathbb{Z}$. 
Therefore, $D_{8 \times 2}(\bm{a}) \in 2^{10} \mathbb{Z}$. 
\end{proof}

\begin{lem}\label{lem:6.2}
Let $p_{i} = a_{i}^{2} + b_{i}^{2} \equiv 1 \pmod{8}$ be a prime with $a_{i} \pm b_{i} \in \left\{ 8 m \pm 1 \mid m \in \mathbb{Z} \right\}$ for each $1 \leq i \leq r$, 
let $p_{r + 1}, \ldots, p_{r + s} \equiv - 1 \pmod{8}$ be primes, 
let $q_{1}, \ldots, q_{t} \equiv 3 \pmod{8}$ be distinct primes, 
and let $k _{1}, \ldots, k_{r + s}$ be non-negative integers. 
Then 
$$
2^{11} p_{1}^{k_{1}} \cdots p_{r}^{k_{r}} p_{r + 1}^{k_{r + 1}} \cdots p_{r + s}^{k_{r + s}} Q \not\in S({\rm C}_{8} \times {\rm C}_{2})
$$
for any $Q \in \left\{ \pm 1, \: \pm q_{1} \cdots q_{t} \right\}$. 
\end{lem}

Let 
\begin{align*}
\alpha_{0} &:= (b_{0} + b_{2})^{2} - (b_{1} + b_{3})^{2}, &
\alpha_{1} &:= (b_{0} - b_{2})^{2} + (b_{1} - b_{3})^{2}, \\ 
\alpha_{2} &:= (d_{0} + d_{2})^{2} - (d_{1} + d_{3})^{2}, &
\alpha_{3} &:= (d_{0} - d_{2})^{2} + (d_{1} - d_{3})^{2}, \\ 
\beta &:= (c_{0}^{2} - c_{2}^{2} + 2 c_{1} c_{3}) - \zeta_{4} (c_{1}^{2} - c_{3}^{2} - 2 c_{0} c_{2}), & 
\gamma &:= (e_{0}^{2} - e_{2}^{2} + 2 e_{1} e_{3}) - \zeta_{4} (e_{1}^{2} - e_{3}^{2} - 2 e_{0} e_{2} ). 
\end{align*}
Then we have 
$\alpha_{0} \alpha_{1} = D_{4}(\bm{b})$, 
$\alpha_{2} \alpha_{3} = D_{4}(\bm{d})$, 
$\beta \overline{\beta} = \widetilde{D}_{4}(\bm{c})$, 
$\gamma \overline{\gamma} = \widetilde{D}_{4}(\bm{e})$, 
%\begin{align*}
%\alpha_{0} \alpha_{1} = D_{4}(\bm{b}), \quad 
%\alpha_{2} \alpha_{3} = D_{4}(\bm{d}), \quad
%\beta \overline{\beta} = \widetilde{D}_{4}(\bm{c}), \quad
%\gamma \overline{\gamma} = \widetilde{D}_{4}(\bm{e}), 
%\end{align*}
where $\overline{x}$ denotes the complex conjugate of $x \in \mathbb{C}$. 
To prove Lemma~$\ref{lem:6.2}$, 
we use the following remark and two lemmas. 

\begin{rem}\label{rem:6.3}
From Remark~$\ref{rem:4.2}$~$(1)$ and 
\begin{align*}
\alpha_{1} \in 2 \mathbb{Z}_{\rm odd} &\iff b_{0} + b_{2} \equiv b_{1} + b_{3} \equiv 1 \pmod{2}, \\ 
\alpha_{3} \in 2 \mathbb{Z}_{\rm odd} &\iff d_{0} + d_{2} \equiv d_{1} + d_{3} \equiv 1 \pmod{2}, \\ 
\beta \overline{\beta} \in 2 \mathbb{Z}_{\rm odd} &\iff c_{0} + c_{2} \equiv c_{1} + c_{3} \equiv 1 \pmod{2}, \\ 
\gamma \overline{\gamma} \in 2 \mathbb{Z}_{\rm odd} &\iff e_{0} + e_{2} \equiv e_{1} + e_{3} \equiv 1 \pmod{2}, 
\end{align*}
we have 
$\alpha_{1} \in 2 \mathbb{Z}_{\rm odd} \iff \alpha_{3} \in 2 \mathbb{Z}_{\rm odd} \iff \beta \overline{\beta} \in 2 \mathbb{Z}_{\rm odd} \iff \gamma \overline{\gamma} \in 2 \mathbb{Z}_{\rm odd}$. 
%$$
%\alpha_{1} \in 2 \mathbb{Z}_{\rm odd} \iff \alpha_{3} \in 2 \mathbb{Z}_{\rm odd} \iff \beta \overline{\beta} \in 2 \mathbb{Z}_{\rm odd} \iff \gamma \overline{\gamma} \in 2 \mathbb{Z}_{\rm odd} 
%$$
\end{rem}

\begin{lem}\label{lem:6.4}
If $D_{8 \times 2}(\bm{a}) \in 2^{11} \mathbb{Z}_{\rm odd}$, 
then we have $(\alpha_{i}, \alpha_{j}) \in 2^{3} \mathbb{Z}_{\rm odd} \times 2^{4} \mathbb{Z}_{\rm odd}$ and 
$\alpha_{1}$, $\alpha_{3}$, $\beta \overline{\beta}$, $\gamma \overline{\gamma} \in 2 \mathbb{Z}_{\rm odd}$, 
where $\{ i, j \} = \{ 0, 2 \}$. 
\end{lem}
\begin{proof}
Let $D_{8 \times 2}(\bm{a}) = \alpha_{0} \alpha_{1} \alpha_{2} \alpha_{3} \beta \overline{\beta} \gamma \overline{\gamma} \in 2^{11} \mathbb{Z}_{\rm odd}$. 
Then, from Lemma~$\ref{lem:4.3}$, $\alpha_{0} \alpha_{1} \equiv \alpha_{2} \alpha_{3} \equiv \beta \overline{\beta} \equiv \gamma \overline{\gamma} \equiv 0 \pmod{2}$ holds. 
In particular, $\alpha_{0} \alpha_{1}, \alpha_{2} \alpha_{3} \in S({\rm C}_{4}) \cap 2 \mathbb{Z} = 2^{4} \mathbb{Z}$. 
From this and Remark~$\ref{rem:6.3}$, we have $\beta \overline{\beta}, \gamma \overline{\gamma} \in 2 \mathbb{Z}_{\rm odd}$. 
Therefore, $\alpha_{1}, \alpha_{3} \in 2 \mathbb{Z}_{\rm odd}$. 
\end{proof}

\begin{lem}\label{lem:6.5}
Let $b_{0} + b_{2} \equiv b_{1} + b_{3} \equiv 1 \pmod{2}$. 
Then the following hold: 
\begin{enumerate}
\item[$(1)$] $\alpha_{0} \equiv \alpha_{1} + 4 (b_{0} b_{2} + b_{1} b_{3}) - 2 \pmod{16}$; 
\item[$(2)$] $\alpha_{2} \equiv \alpha_{3} + 4 (d_{0} d_{2} + d_{1} d_{3}) - 2 \pmod{16}$; 
\item[$(3)$] $\Re(\beta) \equiv (- 1)^{c_{2}} + 2 (c_{0} c_{2} + c_{1} c_{3}) \pmod{8}$; 
\item[$(4)$] $\Re(\gamma) \equiv (- 1)^{e_{2}} + 2 (e_{0} e_{2} + e_{1} e_{3}) \pmod{8}$; 
\item[$(5)$] $(b_{0} b_{2} + b_{1} b_{3}) + (c_{0} c_{2} + c_{1} c_{3}) + (d_{0} d_{2} + d_{1} d_{3}) + (e_{0} e_{2} + e_{1} e_{3}) \equiv 0 \pmod{4}$. 
\end{enumerate}
\end{lem}
\begin{proof}
We obtain (1) from $\alpha_{0} = \alpha_{1} + 4 (b_{0} b_{2} + b_{1} b_{3}) - 2 (b_{1} + b_{3})^{2}$. 
%\begin{align*}
%\alpha_{0} 
%&= (b_{0} + b_{2})^{2} - (b_{1} + b_{3})^{2} \\ 
%&= (b_{0} - b_{2})^{2} + (b_{1} - b_{3})^{2} + 4 b_{0} b_{2} - 2 (b_{1}^{2} + b_{3}^{2}) \\ 
%&= \alpha_{1} + 4 (b_{0} b_{2} + b_{1} b_{3}) - 2 (b_{1} + b_{3})^{2} \\
%&\equiv \alpha_{1} + 4 (b_{0} b_{2} + b_{1} b_{3}) - 2 \pmod{16}. 
%\end{align*}
In the same way, we can obtain (2). 
We obtain (3) from $\Re(\beta) = (c_{0} - c_{2})^{2} - 2 c_{2}^{2} + 2 (c_{0} c_{2} + c_{1} c_{3})$. 
%\begin{align*}
%\Re(\beta) 
%&= c_{0}^{2} - c_{2}^{2} + 2 c_{1} c_{3} \\ 
%&= (c_{0} - c_{2})^{2} - 2 c_{2}^{2} + 2 (c_{0} c_{2} + c_{1} c_{3}) \\ 
%&\equiv 1 - 2 c_{2}^{2} + 2 (c_{0} c_{2} + c_{1} c_{3}) \pmod{8} \\ 
%&\equiv (- 1)^{c_{2}} + 2 (c_{0} c_{2} + c_{1} c_{3}) \pmod{8}. 
%\end{align*}
In the same way, we can obtain (4). 
We prove (5). 
There are four cases: 
$$
\bm{b} \equiv (0, 0, 1, 1), \: (0, 1, 1, 0), \: (1, 1, 0, 0) \: \: \text{or} \: \: (1, 0, 0, 1) \pmod{2}. 
$$
If $\bm{b} \equiv (0, 0, 1, 1) \pmod{2}$, 
then 
\begin{align*}
&(b_{0} b_{2} + b_{1} b_{3}) + (c_{0} c_{2} + c_{1} c_{3}) + (d_{0} d_{2} + d_{1} d_{3}) + (e_{0} e_{2} + e_{1} e_{3}) \\ 
&\qquad \equiv (b_{0} + b_{1}) + (c_{0} + c_{1}) + (d_{0} + d_{1}) + (e_{0} + e_{1}) \\ 
&\qquad \equiv 0 \pmod{4}
\end{align*}
from Remark~$\ref{rem:4.2}$. 
In the same way, 
the remaining three cases can also be proved. 
\end{proof}

\begin{proof}[Proof of Lemma~$\ref{lem:6.2}$]
We prove by contradiction. 
Assume that there exist $a_{0}, a_{1}, \ldots, a_{15} \in \mathbb{Z}$ satisfying 
$D_{8 \times 2}(\bm{a}) = \alpha_{0} \alpha_{1} \alpha_{2} \alpha_{3} \beta \overline{\beta} \gamma \overline{\gamma} = 2^{11} p_{1}^{k_{1}} \cdots p_{r}^{k_{r}} p_{r + 1}^{k_{r + 1}} \cdots p_{r + s}^{k_{r + s}} Q$, 
where $Q$ is $\pm 1$ or  $\pm q_{1} \cdots q_{t}$. 
Since $\alpha_{1}$, $\alpha_{3}$, $\beta \overline{\beta}$ and $\gamma \overline{\gamma}$ are integers expressible in the form $x^{2} + y^{2}$, 
in the prime factorization of them, 
every prime of the form $4 k + 3$ occurs an even number of times. 
From this fact and Lemma~$\ref{lem:6.4}$, 
there exist $l_{f}, m_{f}, n_{f}, u_{f}, v_{f} \geq 0$ satisfying 
\begin{align*}
\alpha_{i} &= 2^{3} p_{1}^{k_{1} - l_{1} - w_{1}} \cdots p_{r}^{k_{r} - l_{r} - w_{r}} p_{r + 1}^{k_{r + 1} - l_{r + 1} - 2 w_{r + 1}} \cdots p_{r + s}^{k_{r + s} - l_{r + s} - 2 w_{r + s}} Q_{1}, \\ 
\alpha_{j} &= 2^{4} p_{1}^{l_{1}} \cdots p_{r}^{l_{r}} p_{r + 1}^{l_{r + 1}} \cdots p_{r + s}^{l_{r + s}} Q_{2}, \\ 
\alpha_{1} &= 2 p_{1}^{m_{1}} \cdots p_{r}^{m_{r}} p_{r + 1}^{2 m_{r + 1}} \cdots p_{r + s}^{2 m_{r + s}}, \\ 
\alpha_{3} &= 2 p_{1}^{n_{1}} \cdots p_{r}^{n_{r}} p_{r + 1}^{2 n_{r + 1}} \cdots p_{r + s}^{2 n_{r + s}}, \\ 
\beta \overline{\beta} &= 2 p_{1}^{u_{1}} \cdots p_{r}^{u_{r}} p_{r + 1}^{2 u_{r + 1}} \cdots p_{r + s}^{2 u_{r + s}}, \\ 
\gamma \overline{\gamma} &= 2 p_{1}^{v_{1}} \cdots p_{r}^{v_{r}} p_{r + 1}^{2 v_{r} + 1} \cdots p_{r + s}^{2 v_{r + s}}, 
\end{align*}
where $\{ i, j \} = \{ 0, 2 \}$, $w_{f} := m_{f} + n_{f} + u_{f} + v_{f}$ and $Q_{1} Q_{2} = Q$. 
From the above, 
we have $\alpha_{i} \equiv 8, \alpha_{j} \equiv 0, \alpha_{1} \equiv \alpha_{3} \equiv 2 \pmod{16}$. 
Therefore, from Lemma~$\ref{lem:6.5}$~(1) and (2), 
$$
(b_{0} b_{2} + b_{1} b_{3}, d_{0} d_{2} + d_{1} d_{3}) \equiv (2, 0) \: \: \text{or} \: \: (0, 2) \pmod{4}. 
$$
Note that $c_{0} c_{2} + c_{1} c_{3} \equiv e_{0} e_{2} + e_{1} e_{3} \equiv 0 \pmod{2}$ since $b_{0} + b_{2} \equiv b_{1} + b_{3} \equiv 1 \pmod{2}$. 
From \cite[Lemma~4.8]{Integer16}, 
we have $\Re(\beta), \Re(\gamma) \in \{ 8 m \pm 1 \mid m \in \mathbb{Z} \}$. 
From this and Lemma~$\ref{lem:6.5}$~(3) and (4), it follows that $c_{0} c_{2} + c_{1} c_{3} \equiv e_{0} e_{2} + e_{1} e_{3} \equiv 0 \pmod{4}$. 
Therefore, we have 
$$
(b_{0} b_{2} + b_{1} b_{3}) + (c_{0} c_{2} + c_{1} c_{3}) + (d_{0} d_{2} + d_{1} d_{3}) + (e_{0} e_{2} + e_{1} e_{3}) \equiv 2 \pmod{4}. 
$$
This contradicts Lemma~$\ref{lem:6.5}$~$(5)$. 
\end{proof}

\subsection{Possible numbers}

Lemmas~$\ref{lem:5.1}$, $\ref{lem:6.1}$ and $\ref{lem:6.2}$ imply that $S \left( {\rm C}_{8} \times {\rm C}_{2} \right)$ does not include every integer that is not mentioned in Lemmas~$\ref{lem:7.1}$ and $\ref{lem:7.2}$. 

\begin{lem}\label{lem:7.1}
For any $m, n \in \mathbb{Z}$, 
the following are elements of $S \left( {\rm C}_{8} \times {\rm C}_{2} \right)$: 
\begin{enumerate}
\item[$(1)$] $16 m + 1$; 
\item[$(2)$] $(16 m - 3) (16 n - 3)$; 
\item[$(3)$] $(16 m + 5) (16 n + 5)$; 
\item[$(4)$] $2^{10} (2 m + 1)$; 
\item[$(5)$] $2^{12} (2 m + 1)$; 
\item[$(6)$] $2^{12}(2 m)$. 
\end{enumerate}
\end{lem}

\begin{lem}\label{lem:7.2}
The following hold: 
\begin{enumerate}
\item[$(1)$] Suppose that $p$ is a prime with $p \equiv - 3 \pmod{8}$, then $2^{11} p (2 m + 1) \in S({\rm C}_{8} \times {\rm C}_{2})$; 
\item[$(2)$] Suppose that $p$ is a prime with $p \equiv 3 \pmod{8}$, then $2^{11} p^{2} (2 m + 1) \in S({\rm C}_{8} \times {\rm C}_{2})$; 
\item[$(3)$] Suppose that $p$ is a prime with $p = a^{2} + b^{2} \equiv 1 \pmod{8}$ and $a + b \equiv \pm 3 \pmod{8}$, then $2^{11} p (2 m + 1) \in S({\rm C}_{8} \times {\rm C}_{2})$. 
\end{enumerate}
\end{lem}

\begin{proof}[Proof of Lemma~$\ref{lem:7.1}$]
We obtain (1) from $D_{8 \times 2}(m + 1, m, \ldots, m) = 16 m + 1$. 
%\begin{align*}
%D_{8 \times 2}(m + 1, m, \ldots, m) 
%&= D_{8}(2m + 1, 2m, \ldots, 2m) D_{8}(1, 0, \ldots, 0) \\ 
%&= D_{4}(4m + 1, 4m, 4m, 4m) D_{4}(1, 0, 0, 0) \widetilde{D}_{4}(1, 0, 0, 0)^{2} \\ 
%&= (8m + 1)^{2} - (8m)^{2} \\ 
%&= 16 m + 1. 
%\end{align*}
From 
\begin{align*}
&D_{8 \times 2}(\overbrace{m + n, \ldots, m + n}^{5}, m + n - 1, m + n - 1, m + n - 1, m - n, \ldots, m - n) \\
&\quad= (16 m - 3) (16 n - 3), \\
&D_{8 \times 2}(\overbrace{m + n + 1, \ldots, m + n + 1}^{5}, m + n, m + n, m + n, m - n, \ldots, m - n) \\ 
&\quad= (16 m + 5) (16 n + 5), 
\end{align*}
we obtain (2) and (3), respectively. 
%\begin{align*}
%&D_{8 \times 2}(\overbrace{m + n, \ldots, m + n}^{5}, m + n - 1, m + n - 1, m + n - 1, m - n, \ldots, m - n) \\ 
%&\quad= D_{8}(2 m, \ldots, 2 m, 2 m - 1, 2 m - 1, 2 m - 1) D_{8}(2 n, \ldots, 2 n, 2 n - 1,  2 n - 1, 2 n - 1) \\ 
%&\quad= D_{4}(4 m, 4 m - 1, 4 m - 1, 4 m - 1) D_{4}(4 n, 4 n - 1, 4 n - 1, 4 n - 1) \widetilde{D}_{4}(0, 1, 1, 1)^{2} \\ 
%&\quad= \left\{ (8 m - 1)^{2} - (8 m - 2)^{2} \right\} \left\{ (8 n - 1)^{2} - (8 n - 2)^{2} \right\} \\ 
%&\quad= (16 m - 3) (16 n - 3). 
%\end{align*}
%\begin{align*}
%&D_{8 \times 2}(\overbrace{m + n + 1, \ldots, m + n + 1}^{5}, m + n, m + n, m + n, m - n, \ldots, m - n) \\ 
%&\quad= D_{8}(2 m + 1, \ldots, 2 m + 1, 2 m, 2 m, 2 m) D_{8}(2 n + 1, \ldots,  2 n + 1, 2 n, 2 n, 2 n) \\ 
%&\quad= D_{4}(4 m + 2, 4 m + 1, 4 m + 1, 4 m + 1) D_{4}(4 n + 2, 4 n + 1, 4 n + 1, 4 n + 1) \widetilde{D}_{4}(0, 1, 1, 1)^{2} \\ 
%&\quad= \left\{ (8 m + 3)^{2} - (8 m + 2)^{2} \right\} \left\{ (8 n + 3)^{2} - (8 n + 2)^{2} \right\} \\ 
%&\quad= (16 m + 5) (16 n + 5). 
%\end{align*}
We obtain (4) from 
\begin{align*}
&D_{8 \times 2}(m + 1, m + 1, m + 1, m, m, m, m + 1, m, \ldots, m) = 2^{10} (4 m + 1), \\
&D_{8 \times 2}(\overbrace{m, \ldots, m}^{6}, m + 1, m - 1, m, m, m - 1, m, m - 1, m - 1, m, m - 1) = 2^{10} (4 m - 1).  
\end{align*}
%\begin{align*}
%&D_{8 \times 2}(m + 1, m + 1, m + 1, m, m, m, m + 1, m, \ldots, m) \\ 
%&\quad= D_{8}(2 m + 1, 2 m + 1, 2 m + 1, 2 m, 2 m, 2 m, 2 m + 1, 2 m) D_{8}(1, 1, 1, 0, 0, 0, 1, 0) \\ 
%&\quad= D_{4}(4 m + 1, 4 m + 1, 4 m + 2, 4 m) D_{4}(1, 1, 2, 0) \widetilde{D}_{4}(1, 1, 0, 0)^{2} \\ 
%&\quad= 2 \left\{ (8 m + 3)^{2} - (8 m + 1)^{2} \right\} \cdot 2^{4} \cdot 2^{2} \\ 
%&\quad= 2^{7} (32 m + 8) \\ 
%&\quad= 2^{10} (4 m + 1), \\
%&D_{8 \times 2}(\overbrace{m, \ldots, m}^{6}, m + 1, m - 1, m, m, m - 1, m, m - 1, m - 1, m, m - 1) \\ 
%&\quad = D_{8}(2 m, 2 m, 2 m - 1, 2 m, 2 m - 1, 2 m - 1, 2 m + 1, 2 m - 2) D_{8}(0, 0, 1, 0, 1, 1, 1, 0) \\ 
%&\quad = D_{4}(4 m - 1, 4 m - 1, 4 m, 4 m - 2) \widetilde{D}_{4}(1, 1, - 2, 2) D_{4}(1, 1, 2, 0) \widetilde{D}_{4}(- 1, - 1, 0, 0) \\ 
%&\quad = 2 \left\{ (8 m - 1)^{2} - (8 m - 3)^{2} \right\} \cdot 2 \cdot 2^{4} \cdot 2 \\ 
%&\quad = 2^{7} (32 m - 8) \\ 
%&\quad = 2^{10} (4 m - 1).  
%\end{align*}
We obtain (5) from $D_{8 \times 2}(m + 2, m, \overbrace{m + 1, \ldots, m + 1}^{6}, m, \ldots, m) = 2^{12} (2 m + 1)$. 
%\begin{align*}
%&D_{8 \times 2}(m + 2, m, \overbrace{m + 1, \ldots, m + 1}^{6}, m, \ldots, m) \\ 
%&\quad = D_{8}(2 m + 2, 2 m, 2 m + 1, \ldots, 2 m + 1) D_{8}(2, 0, 1, \ldots, 1) \\ 
%&\quad = D_{4}(4 m + 3, 4 m + 1, 4 m + 2, 4 m + 2) D_{4}(3, 1, 2, 2) \widetilde{D}_{4}(1, - 1, 0, 0)^{2} \\ 
%&\quad = 2 \left\{ (8 m + 5)^{2} - (8 m + 3)^{2} \right\} \cdot 2^{5} \cdot 2^{2} \\ 
%&\quad = 2^{8} (32 m + 16) \\ 
%&\quad = 2^{12} (2 m + 1). 
%\end{align*}
From 
\begin{align*}
&D_{8 \times 2}(m + 1, m, m, m + 1, m + 1, m, m + 1, m, m - 1, m - 1, m, m - 1, m, m, m - 1, m) \\ 
&\quad = 2^{12} (2 m), 
\end{align*}
%\begin{align*}
%&D_{8 \times 2}(m + 1, m, m, m + 1, m + 1, m, m + 1, m, m - 1, m - 1, m, m - 1, m, m, m - 1, m) \\ 
%&\quad = D_{8}(2 m, 2 m - 1, 2 m, 2 m, 2 m + 1, 2 m, 2 m, 2 m) D_{8}(2, 1, 0, 2, 1, 0, 2, 0) \\ 
%&\quad = D_{4}(4 m + 1, 4 m - 1, 4 m, 4 m) \widetilde{D}_{4}(- 1, - 1, 0, 0) D_{4}(3, 1, 2, 2) \widetilde{D}_{4}(1, 1, - 2, 2) \\ 
%&\quad = 2 \left\{ (8 m + 1)^{2} - (8 m - 1)^{2} \right\} \cdot 2 \cdot 2^{5} \cdot 2 \\ 
%&\quad = 2^{8} (32 m) \\ 
%&\quad = 2^{12} (2 m). 
%\end{align*}
we obtain (6). 
\end{proof}

To prove Lemma~$\ref{lem:7.2}$, 
we use the following lemma. 

\begin{lem}[{\cite[Proof of Theorem~5.1]{Integer16}}]\label{lem:7.3}
The following hold: 
\begin{enumerate}
\item[$(1)$] Suppose that $p \equiv - 3 \pmod{8}$ is a prime, then there exist $k, l \in \mathbb{Z}$ satisfying $$2 p = (8 k + 3)^{2} + (8 l + 1)^{2}; $$ 
\item[$(2)$] Suppose that $p \equiv 3 \pmod{8}$ is a prime, then there exist $k, l \in \mathbb{Z}$ satisfying $$p = (4 k - 1)^{2} + 2 (4 l - 1)^{2}; $$
\item[$(3)$] Suppose that $p = a^{2} + b^{2} \equiv 1 \pmod{8}$ is a prime with $a + b \equiv \pm 3 \pmod{8}$, then there exist $k, l, m, n \in \mathbb{Z}$ satisfying 
\begin{align*}
&2 p = \left\{ (4 k - 1)^{2} - (4 m - 2)^{2} + 2 (2 l - 1) (4 n) \right\}^{2} \\ 
&\qquad + \left\{ (2 l - 1)^{2} - (4 n)^{2} - 2 (4 k - 1) (4 m - 2) \right\}^{2}. 
\end{align*}
\end{enumerate}
\end{lem}

\begin{proof}[Proof of Lemma~$\ref{lem:7.2}$]
First, we prove (1). 
Let 
\begin{align*}
a_{0} &= k + m + 2, 
&a_{1} &= l + m + 1, 
&a_{2} &= - k + m, 
&a_{3} &= - l + m + 1, \\
a_{4} &= k + m + 1, 
&a_{5} &= l + m, 
&a_{6} &= - k + m + 1, 
&a_{7} &= - l + m, \\
a_{8} &= k - m, 
&a_{9} &= l - m, 
&a_{10} &= - k - m, 
&a_{11} &= - l - m - 1, \\
a_{12} &= k - m, 
&a_{13} &= l - m, 
&a_{14} &= - k - m - 1, 
&a_{15} &= - l - m. 
\end{align*}
%\begin{align*}
%a_{0} &= k + m + 2, \: \: 
%a_{1} = l + m + 1, \: \: 
%a_{2} = - k + m, \: \: 
%a_{3} = - l + m + 1, \: \: \\
%a_{4} &= k + m + 1, \: \: 
%a_{5} = l + m, \: \: 
%a_{6} = - k + m + 1, \: \: 
%a_{7} = - l + m, \: \: \\
%a_{8} &= k - m, \: \: 
%a_{9} = l - m, \: \: 
%a_{10} = - k - m, \: \: 
%a_{11} = - l - m - 1, \: \: \\
%a_{12} &= k - m, \: \: 
%a_{13} = l - m, \: \: 
%a_{14} = - k - m - 1, \: \: 
%a_{15} = - l - m. 
%\end{align*}
Then $D_{8 \times 2}(\bm{a}) = 2^{10} \left\{ (8 k + 3)^{2} + (8 l + 1)^{2} \right\} (2 m + 1)$. 
%\begin{align*}
%D_{8 \times 2}(\bm{a}) 
%&= D_{8}( 2 k + 2, 2 l + 1, - 2 k, - 2 l, 2 k + 1, 2 l, - 2 k, - 2 l ) \\ 
%&\qquad \times D_{8}( 2 m + 2, 2 m + 1, 2 m, 2 m + 2, 2 m + 1, 2 m, 2 m + 2, 2 m ) \\ 
%&= D_{4}(4k + 3, 4 l + 1, - 4 k, - 4 l) \widetilde{D}_{4}(1, 1, 0, 0) \\ 
%&\qquad \times D_{4}(4 m + 3, 4 m + 1, 4 m + 2, 4 m + 2) \widetilde{D}_{4}(1, 1, - 2, 2) \\ 
%&= 2^{3} \left\{ (8 k + 3)^{2} + (8 l + 1)^{2} \right\} \cdot 2 \cdot 2 \left\{ (8 m + 5)^{2} - (8 m + 3)^{2} \right\} \cdot 2 \\ 
%&= 2^{6} \left\{ (8 k + 3)^{2} + (8 l + 1)^{2} \right\} (32 m + 16) \\ 
%&= 2^{10} \left\{ (8 k + 3)^{2} + (8 l + 1)^{2} \right\} (2 m + 1). 
%\end{align*}
Therefore, from Lemma~$\ref{lem:7.3}$~(1), we obtain (1). 
We prove (2). 
Let 
\begin{align*}
a_{0} &= k + l + m, 
&a_{1} &= k - l + m, 
&a_{2} &= - l + m + 1, 
&a_{3} &= - l + m + 1, \\ 
a_{4} &= - k - l + m + 2, 
&a_{5} &= - k + l + m + 1, 
&a_{6} &= l + m + 1, 
&a_{7} &= l + m, \\ 
a_{8} &= k + l - m, 
&a_{9} &= k - l - m, 
&a_{10} &= - l - m, 
&a_{11} &= - l - m, \\ 
a_{12} &= - k - l - m, 
&a_{13} &= - k + l - m - 1, 
&a_{14} &= l - m - 1, 
&a_{15} &= l - m 
\end{align*}
%\begin{align*}
%a_{0} &= k + l + m, \: \: 
%a_{1} = k - l + m, \: \: 
%a_{2} = - l + m + 1, \: \: 
%a_{3} = - l + m + 1, \: \: \\ 
%a_{4} &= - k - l + m + 2, \: \: 
%a_{5} = - k + l + m + 1, \: \: 
%a_{6} = l + m + 1, \: \: 
%a_{7} = l + m, \: \: \\ 
%a_{8} &= k + l - m, \: \: 
%a_{9} = k - l - m, \: \: 
%a_{10} = - l - m, \: \: 
%a_{11} = - l - m, \: \: \\ 
%a_{12} &= - k - l - m, \: \: 
%a_{13} = - k + l - m - 1, \: \: 
%a_{14} = l - m - 1, \: \: 
%a_{15} = l - m 
%\end{align*}
and $s := 4 k - 1$, $t := 4 l - 1$. 
Then $D_{8 \times 2}(\bm{a}) = 2^{11} (s^{2} + 2 t^{2})^{2} (2 m + 1)$. 
%\begin{align*}
%D_{8 \times 2}(\bm{a}) 
%&= D_{8}( 2 k + 2 l, 2 k - 2 l, - 2 l + 1, - 2 l + 1, - 2 k - 2 l + 2, - 2 k + 2 l, 2 l, 2 l ) \\ 
%&\qquad \times D_{8}( 2 m, 2 m, 2 m + 1, 2 m + 1, 2 m + 2, 2 m + 2, 2 m + 2, 2 m ) \\ 
%&= D_{4}(2, 0, 1, 1) \widetilde{D}_{4}(s + t, s - t, - t, - t) \\ 
%&\qquad \times D_{4}(4 m + 2, 4 m + 2, 4 m + 3, 4 m + 1) \widetilde{D}_{4}(- 2, - 2, - 1, 1) \\ 
%&= 2^{4} \left\{ \left( (s + t)^{2} - t^{2} - 2 (s - t) t \right)^{2} + \left( (s - t)^{2} - t^{2} + 2 (s + t) t \right)^{2} \right\} \\ 
%&\qquad \times 2 \left\{ (8 m + 5)^{2} - (8 m + 3)^{2} \right\} \cdot 2 \\ 
%&= 2^{6} \left\{ (s^{2} + 2 t^{2})^{2} + (s^{2} + 2 t^{2})^{2} \right\} (32 m + 16) \\ 
%&= 2^{11} (s^{2} + 2 t^{2})^{2} (2 m + 1). 
%\end{align*}
Therefore, (2) is proved from Lemma~$\ref{lem:7.3}$~(2). 
We prove (3). 
Let $l' := \frac{l}{2}$ if $l$ is even; $\frac{l - 1}{2}$ if $l$ is odd and 
\begin{align*}
a_{0} &= k + r + 1, 
&a_{1} &= l' + r, 
&a_{2} &= m + r, 
&a_{3} &= n + r, \\
a_{4} &= - k + r + 1, 
&a_{5} &= - l' + r + \frac{(- 1)^{l} + 1}{2}, 
&a_{6} &= - m + r + 2, 
&a_{7} &= - n + r + 1, \\
a_{8} &= k - r - 1, 
&a_{9} &= l' - r, 
&a_{10} &= m - r, 
&a_{11} &= n - r, \\
a_{12} &= - k - r, 
&a_{13} &= - l' - r + \frac{(- 1)^{l} - 1}{2}, 
&a_{14} &= - m - r, 
&a_{15} &= - n - r - 1. 
\end{align*}
%\begin{align*}
%a_{0} &= k + r + 1, \: \: 
%a_{1} = l' + r, \: \: 
%a_{2} = m + r, \: \: 
%a_{3} = n + r, \: \: \\
%a_{4} &= - k + r + 1, \: \: 
%a_{5} = - l' + r + \frac{(- 1)^{l} + 1}{2}, \: \: 
%a_{6} = - m + r + 2, \: \: 
%a_{7} = - n + r + 1, \: \: \\
%a_{8} &= k - r - 1, \: \: 
%a_{9} = l' - r, \: \: 
%a_{10} = m - r, \: \: 
%a_{11} = n - r, \: \: \\
%a_{12} &= - k - r, \: \: 
%a_{13} = - l' - r + \frac{(- 1)^{l} - 1}{2}, \: \: 
%a_{14} = - m - r, \: \: 
%a_{15} = - n - r - 1. 
%\end{align*}
Then $D_{8 \times 2}(\bm{a}) = 2^{10} \widetilde{D}_{4}(4 k - 1, 2 l - 1, 4 m - 2, 4 n) (2 r + 1)$. 
%\begin{align*}
%D_{8 \times 2}(\bm{a}) 
%&= D_{8}(2k, 2l', 2m, 2n, - 2k + 1, - 2l' + (-1)^{l}, - 2m + 2, - 2 n) \\ 
%&\qquad \times D_{8}(2 r + 2, 2 r, 2 r, 2 r, 2 r + 1, 2 r + 1, 2 r + 2, 2 r + 2) \\ 
%&= D_{4}(1, (- 1)^{l}, 2, 0) \widetilde{D}_{4}(4 k - 1, 4 l' - (- 1)^{l}, 4 m - 2, 4 n) \\ 
%&\qquad \times D_{4}(4 r + 3, 4 r + 1, 4 r + 2, 4 r + 2) \widetilde{D}_{4}(1, - 1, - 2, - 2) \\ 
%&= 2^{4} \cdot \widetilde{D}_{4}(4 k - 1, 2 l - 1, 4 m - 2, 4 n) \cdot 2 \left\{ (8 r + 5)^{2} - (8 r + 3)^{2} \right\} \cdot 2 \\ 
%&= 2^{6} \widetilde{D}_{4}(4 k - 1, 2 l - 1, 4 m - 2, 4 n) (32 r + 16) \\ 
%&= 2^{10} \widetilde{D}_{4}(4 k - 1, 2 l - 1, 4 m - 2, 4 n) (2 r + 1). 
%\end{align*}
Therefore, (3) is proved from Lemma~$\ref{lem:7.3}$~(3). 
\end{proof}

From Lemmas~$\ref{lem:5.1}$, $\ref{lem:6.1}$, $\ref{lem:6.2}$, $\ref{lem:7.1}$ and $\ref{lem:7.2}$, Theorem~$\ref{thm:1.5}$ is proved.

\clearpage

\bibliography{reference}
\bibliographystyle{plain}

\medskip
\begin{flushleft}
Faculty of Education, 
University of Miyazaki, 
1-1 Gakuen Kibanadai-nishi, 
Miyazaki 889-2192, 
Japan \\ 
{\it Email address}, Naoya Yamaguchi: n-yamaguchi@cc.miyazaki-u.ac.jp \\
{\it Email address}, Yuka Yamaguchi: y-yamaguchi@cc.miyazaki-u.ac.jp 
\end{flushleft}

\end{document}